\documentclass[11pt]{amsart}
\usepackage{amsthm,amsmath,amssymb}

\usepackage{geometry}
\usepackage{graphicx,cite,enumitem}
\usepackage{color}

\numberwithin{equation}{section}
\numberwithin{figure}{section}
\theoremstyle{plain}
\newtheorem{thm}{Theorem}[section]
\newtheorem{lem}[thm]{Lemma}
\newtheorem{cor}[thm]{Corollary}

\newtheorem{clm}[thm]{Claim}

\theoremstyle{remark}
\newtheorem{rmk}[thm]{Remark}

\newcommand{\M}{\operatorname{M}}

\newcommand{\PP}{\operatorname{PP}}

\begin{document}

\title{A shuffling theorem for lozenge tilings of doubly-dented hexagons}

\author{Tri Lai}
\address{Department of Mathematics, University of Nebraska -- Lincoln, Lincoln, NE 68588, U.S.A.}
\email{tlai3@unl.edu}
\thanks{This research was was supported in part  by Simons Foundation Collaboration Grant (\# 585923).}

\author{Ranjan Rohatgi}
\address{Department of Mathematics and Computer Science, Saint Mary's College, Notre Dame, IN 46556, U.S.A.}
\email{rrohatgi@saintmarys.edu}



\date{\today}

\dedicatory{}

\begin{abstract}
 MacMahon's theorem on plane partitions yields a simple product formula for tiling number of a hexagon, and Cohn, Larsen and Propp's theorem provides an explicit enumeration for tilings of a dented semihexagon via semi-strict Gelfand--Tsetlin patterns. In this paper, we prove a natural hybrid of the two theorems for hexagons with an arbitrary set of  unit triangles removed along the a horizontal axis. In particular, we show that the `shuffling' of removed unit triangles only changes the tiling number of the region by a simple multiplicative factor. Our main result generalizes a number of known enumerations and asymptotic enumerations of tilings. We also reveal connections of the main result to the study of symmetric functions and $q$-series.
\end{abstract}

\maketitle
\section{Introduction}

MacMahon's classical theorem  \cite{Mac} on plane partition fitting in a given box is equivalent to the fact that the number of lozenge tilings of a centrally symmetric hexagon $H(a,b,c)$ of side-lengths $a,b,c,a,b,c$ (in this cyclic order) is given by the simple product:
\begin{equation}\label{Maceq}
\PP(a,b,c):=\prod_{i=1}^{a}\prod_{j=1}^{b}\prod_{k=1}^{c}\frac{i+j+k-1}{i+j+k-2}.
\end{equation}

This formula was generalized by Cohn, Larsen and Propp \cite[Proposition 2.1]{CLP} when they presented a correspondence between lozenge tilings of a semihexagon with unit triangles removed on the base and semi-strict Gelfand-Tsetlin patterns. In particular, the \emph{dented semihexagon} $S_{a,b}(s_1,s_2,\dots,s_a)$ is the region obtained from the upper half of the symmetric hexagon of side-lengths $b,a,a,b,a,a$ (in clockwise order, starting from the north side) by removing $a$ up-pointing unit triangles along the base at the positions $s_1,s_2,\dotsc,s_a$ from left to right. The number of lozenge tilings of the dented semihexagon is given by
\begin{equation}\label{CLPeq}
\M(S_{a,b}(s_1,s_2,\dots,s_a))=\prod_{1\leq i<j \leq a}\frac{s_j-s_i}{j-i},
\end{equation}
where we use the notation $\M(R)$ for the number of lozenge tilings of the region $R$.

In this paper, we consider a hybrid object between MacMahon's hexagon and Cohn--Larsen--Propp's dented semihexagon. Our region is a hexagon on the triangular lattice, as in the case of MacMahon's theorem,  with an arbitrary set of unit triangles removed along a horizontal axis, like the dents in Cohn--Larsen--Propp's theorem (see Fig. \ref{multiplefernfig} A).  In general, the tiling numbers of such regions are not given by  simple product formula. However, we show that their tiling number only changes by a simple multiplicative factor when we shuffle the positions of up- and down-pointing removed triangles (see Theorem \ref{factorization} and its generalizations in Theorems \ref{genfactorization} and \ref{qfactorization}).

Our main theorems imply a number of known tiling enumerations of regions with `holes' (e.g. \cite{Ciu1, Twofern, HoleDent}). Here, a \emph{hole} is a portion removed from a region. We also show that our main theorems can be used to obtain new results in asymptotic enumeration of tilings, including the enumeration of the so-called `\emph{doubly--dented hexagon}' in \cite{Twofern}, the main results of the first author about hexagon with three arrays of triangles removed in \cite[Theorems 2.11 and 2.12]{HoleDent} and Ciucu's main results about `\emph{$F$-cored hexagons}' in \cite[Theorems 1.1 and 2.1]{Ciu1} (see Corollary \ref{cordual}).



\section{Shuffling theorems}
Let $x,y,n,u,d$ be nonnegative integers, such that $u,d \leq n$. Consider a symmetric hexagon of side-lengths $x+n-u,y+u,y+d,x+n-d,y+d,y+u$ in clockwise order, starting from the north side. We remove $u+d$ arbitrary unit triangles along the lattice line $l$ that contains the west and the east vertices of the hexagon. Assume further that, among these $u+d$ removed triangles, there are $u$ up-pointing ones and $d$ down-pointing ones. Let $U=\{s_1,s_2,\dotsc,s_u\}$ and $D=\{t_1,t_2,\dots, t_d\}$ be, respectively, the sets of positions of the up-pointing and down-pointing removed unit triangles (ordered from left to right), such that $|U\cup D|=n$ (i.e., $U,D\subseteq [x+y+n]$, $U$ and $D$ are not necessarily disjoint). Denote by $H_{x,y}(U; D)$ the resulting region . See Fig.  \ref{multiplefernfig} A for an example of such a region and Fig. \ref{multiplefernfig} B for a sample tiling; we ignore the two horizontal unit ``barriers'' at the positions 6 and 13 on $l$ at the moment.

We now consider `shuffling' the up- and down-pointing unit triangles in the symmetric difference $U\Delta D$ to obtain new position sets $U'$ and $D'$ for the up-pointing and down-pointing removed triangles. (In particular, $U$ and $U'$  have the same size, and so do $D$ and $D'$.)  The following theorem shows that the shuffling of removed triangles only changes the tiling number by a simple multiplicative factor. Moreover, the factor can be written in a similar form to Cohn--Larsen--Propp's formula (i.e. the product on the right-hand side of Eq. \ref{CLPeq}).

\begin{figure}\centering
\includegraphics[width=12cm]{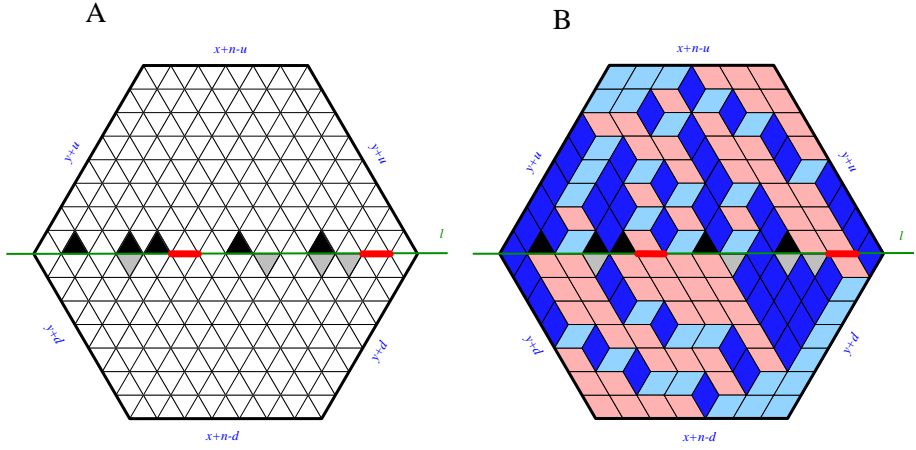}
\caption{(A) The region $H_{4,3}(2,4,5,8,11;\ 4,9,11,12)$ and (B) a lozenge tiling of it. The black and shaded triangles indicate the unit triangles removed.}\label{multiplefernfig}
\end{figure}

\begin{thm}[Shuffling Theorem]\label{factorization} For nonnegative integers $x,y,n,u,d$ ($u,d\leq n$) and four ordered subsets $U=\{s_1,s_2,\dotsc,s_u\}$,  $D=\{t_1,t_2,\dots, t_d\}$, $U'=\{s'_1,s'_2,\dotsc,s'_{u}\}$, and  $D'=\{t'_1,t'_2,\dots, t'_{d}\}$ of $[x+y+n]$  such that $U\cup D =U'\cup D'$, and $U\cap D =U'\cap D'$. Then
\begin{equation}\label{maineq1}
  \frac{\M(H_{x,y}(U; D))}{\M(H_{x,y}(U'; D'))}=\prod_{1\leq i <j\leq u}\frac{s_j-s_i}{s'_j-s'_i}\displaystyle\prod_{1\leq i <j\leq d}\frac{t_j-t_i}{t'_j-t'_i}.\end{equation}
\end{thm}
We would like to emphasize that, in general, the numbers of tilings of two regions on the left-hand side of   Eq. \ref{maineq1} are \emph{not} given by simple product formulas.

\begin{rmk}[A geometrical interpretation] By Cohn--Larson--Propp's theorem (see Eq. \ref{CLPeq}), Eq. \ref{maineq1} in Theorem \ref{factorization} can be written in terms of tiling numbers as
\begin{align}\label{geointereq}
  \frac{\M(H_{x,y}(U; D))}{\M(H_{x,y}(U'; D'))}=\frac{\M(S_{u,x+y+n-u}(U))\M(S_{d,x+y+n-d}(D))}{\M(S_{u,x+y+n-u}(U'))\M(S_{d,x+y+n-d}(D'))}.
\end{align}
 It would be interesting to have a combinatorial explanation for this identity.
 
Moreover, one readily sees that the two dented semihexagons in the numerator of the right-hand side are obtained by dividing the region $H_{x+y,0}(U; D)$ along the horizontal axis $l$. Similarly, the two dented semihexagons in the denominator are obtained by dividing $H_{x+y,0}(U'; D')$ along the horizontal axis. This means that, identity  (\ref{geointereq}) is equivalent to
\begin{equation}\label{geointereq2}
\M(H_{x,y}(U; D))\M(H_{x+y,0}(U'; D'))=\M(H_{x,y}(U'; D'))\M(H_{x+y,0}(U; D)).
\end{equation}
The both sides of (\ref{geointereq2}) count pairs of tilings of doubly-dented hexagons. It would be interesting to find a bijective proof for this identity.
\end{rmk}


We can generalize our Shuffling Theorem \ref{factorization} by additionally allowing the unit triangles in the symmetric difference $U \Delta D$ to `flip' (from up-pointing to down-pointing, and vice versa). It is possible, then, that the new position sets of removed up-pointing triangles and removed down-pointing triangles $U'$ and $D'$ may have sizes \emph{different} than those of $U$ and $D$.
We also allow the appearance of ``\emph{barriers}'' along the axis $l$. A barrier is a unit horizontal lattice interval which is not allowed to be contained within a lozenge in a tiling. Assume that we have a set of barriers at the positions $B\subseteq [x+y+n]-U\cup D$ so that vertical lozenges may not appear at the positions in $B$ and that $|B|\leq x$ (see the red barriers in Fig.  \ref{multiplefernfig}; $B=\{6,13\}$ in this case). We now consider the tilings of $H_{x,y}(U;\ D)$  which are compatible with the set of barriers $B$. Denote by $H_{x,y}(U;D;B)$ the doubly-dented hexagons with such setup of removed unit triangles and barrier. 
\begin{thm}[Generalized Shuffling Theorem]\label{genfactorization} For nonnegative integers $x,y,n,u,d$ ($u,d\leq n$) and five ordered subsets $U=\{s_1,s_2,\dotsc,s_u\}$,  $D=\{t_1,t_2,\dots, t_d\}$, $U'=\{s'_1,s'_2,\dotsc,s'_{u'}\}$,  $D'=\{t'_1,t'_2,\dots, t'_{d'}\}$, and $B:=\{k_1,k_2,\dots,k_b\}$ of $[x+y+n]$  such that $U\cup D =U'\cup D'$, $U\cap D =U'\cap D'$, $B\cup (U\cup D)=\emptyset$, and $|B|\leq x$, we always have
\small{\begin{equation}\label{genmaineq}
  \frac{\M(H_{x,y}(U;D;B))}{\M(H_{x,y}(U';D';B))}= \frac{\displaystyle\prod_{1\leq i <j\leq u}\frac{s_j-s_i}{j-i}\displaystyle\prod_{1\leq i <j\leq d}\frac{t_j-t_i}{j-i}\PP(u,d,y)}{\displaystyle\prod_{1\leq i <j\leq u'}\frac{s'_j-s'_i}{j-i}\displaystyle\prod_{1\leq i <j\leq d'}\frac{t'_j-t'_i}{j-i}\PP(u'd',y)}.
\end{equation}}
\end{thm}
An surprising fact is that the right-hand side of Eq. \ref{genmaineq}  does \emph{not} depend on the barrier set $B$. It would be quite interesting to explain this phenomenon combinatorially.

We note that the above generalized shuffling theorem can be viewed as the enumeration of certain `restricted tilings' in the case when $U\cap D=\emptyset$. Indeed, the tilings of $H_{x,y}(U;\ D;\ B)$ are in bijection with tilings of $H_{x,y+(u+d-n)}(U\setminus D; D\setminus U)$ which contain a fixed vertical lozenge at each position in $U \cap D$ and which do not contain a vertical lozenge at each position in $B$. These types of tiling enumerations were investigated by Fischer \cite{Fischer} and Fulmek and Krattenthaler in \cite{FK1,FK2}.

\begin{figure}\centering
\includegraphics[width=7cm]{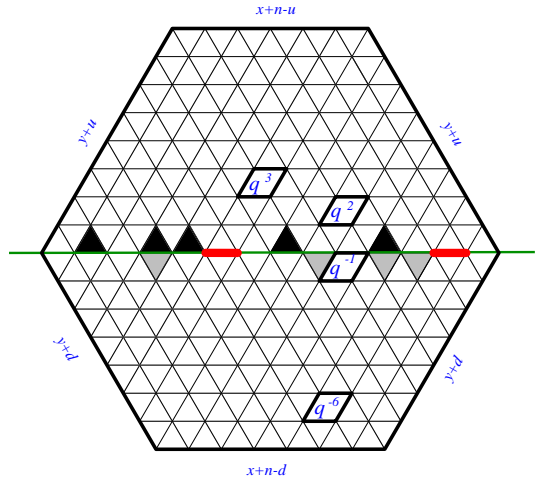}
\caption{Assigning weights to lozenges in a doubly-dented hexagon.}\label{weight}
\end{figure}

We can assign to each right-tilting lozenge above $l$ (resp., below $l$) a weight $q^{z}$ (resp, a weight $q^{-z}$), where $(z-\frac{1}{2})\frac{\sqrt{3}}{2}$ is the distance from its center to the axis $l$ (see Figure \ref{weight}). The `\emph{tiling-generating function}' of a region $H_{x,y}(U;D; B)$ is the sum of the weights of all tilings of the region, where the \emph{weight} of a tiling is the product of weights of all its constituent lozenges. We denote this tiling-generating function by $\M_q(H_{x,y}(U;D;B))$. We have a  $q$-analog of Theorem \ref{genfactorization}:
\begin{thm}[$q$-Shuffling Theorem]\label{qfactorization} With the same assumptions as in Theorem \ref{genfactorization}, we have
\begin{align}\label{qmianeq}
  \frac{\M_q(H_{x,y}(U;D;B))}{\M_q(H_{x,y}(U';D';B))}=q^C \frac{\displaystyle\prod_{1\leq i <j\leq u}\frac{q^{s_j}-q^{s_i}}{q^{j}-q^{i}}\displaystyle\prod_{1\leq i <j\leq d}\frac{q^{t_j}-q^{t_i}}{j-i}\PP_q(u,d,y)}{\displaystyle\prod_{1\leq i <j\leq u'}\frac{q^{s'_j}-q^{s'_i}}{q^{j}-q^{i}}\displaystyle\prod_{1\leq i <j\leq d'}\frac{q^{t'_j}-q^{t'_i}}{q^{j}-q^{i}}\PP_q(u'd',y)}.
   \end{align}
where 
\begin{equation}
\PP_q(a,b,c):=\prod_{i=1}^{a}\prod_{j=1}^{b}\prod_{k=1}^{c}\frac{1-q^{i+j+k-1}}{1-q^{i+j+k-2}}.
\end{equation}
and 
where
\begin{align*}
C=(d-x-n)\binom{y+d+1}{2}-(d'-x-n)\binom{y+d'+1}{2}+udy-u'd'y.
\end{align*}
\end{thm}
We note that MacMahon's Theorem \cite{Mac} states that the number of plane partitions fitting in an $a\times b \times c$- box of given by $\PP_q(a,b,c)=\prod_{i=1}^{a}\prod_{j=1}^{b}\prod_{k=1}^{c}\frac{1-q^{i+j+k-1}}{1-q^{i+j+k-2}}$. The formula (\ref{Maceq}) follows by letting $q \to 1$.

\begin{rmk}
We note that our shuffling theorems (Theorems \ref{factorization}, \ref{genfactorization}, and \ref{qfactorization}) above only for symmetric hexagons with removed unit triangles and barriers (we call them generally the `\emph{obstacles}') running along the horizontal axis $l$ that passes the east and west vertices of the hexagons. One would ask for similar results for general hexagons (not necessarily symmetric) with obstacles on an arbitrary horizontal lattice line (not necessarily passing a vertex of the hexagon). However, the latter case can be implied from our main theorems via `\emph{forced lozenges}'.  A \emph{forced lozenge} in a region $R$ is a lozenge that appears in every tiling of $R$. In the unweighted case, the removal of forced lozenges does not change the tiling number of the region. We consider a general hexagons of side-lengths $x+n-u, y+u, z+d,x+n-d,y+d,z+u$ (in the clockwise order from the north side), and we would like removing unit triangles and placing barriers along an \emph{arbitrary} horizontal lattice line $l$. There are four possible cases depending on the relative positions of the axis $l$ and the east and the west vertices of the hexagon. Figure \ref{forced2} shows that, in all four cases, by removing forced lozenges from a symmetric doubly-dented hexagon as considered in our shuffling theorems, we can obtain general doubly-dented hexagon with an arbitrary position of the axis $l$. This means that one can obtain similar shuffling theorems for general hexagons with obstacles on an arbitrary axis from our shuffling theorems.
\end{rmk}

\begin{figure}\centering
\includegraphics[width=12cm]{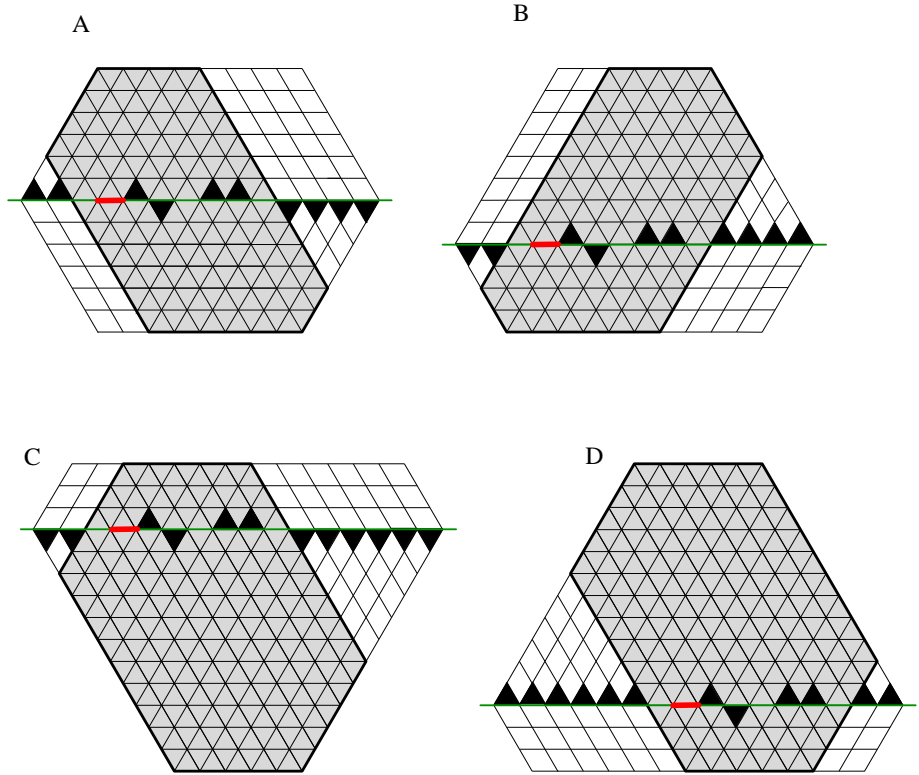}
\caption{Obtaining general doubly-dented hexagons from symmetric doubly-dented hexagons by removing forced lozenges.}\label{forced2}
\end{figure}

\section{An asymptotic enumeration}

 Ciucu and Krattenthaler in \cite{CK} proved a counterpart of MacMahon's theorem (Eq. \ref{Maceq}) by obtaining the asymptotic tiling number of the exterior of a concave hexagonal contour in which we turn $120^{\circ}$ after drawing each edge (see Fig. \ref{contour} B; the tiling number in MacMahon's theorem is for the interior of a hexagonal contour in which each turn $60^{\circ}$ after drawing each edge as in Fig. \ref{contour} A). Ciucu \cite{Ciu1} later obtained a similar counterpart of Cohn--Larsen--Propp's theorem corresponding to the exterior of a concave polygon with an arbitrary number sides (see the contour in Fig. \ref{contour} C). Recently, the first author generalized the asymptotic result of Ciucu to the union of three polygons \cite{HoleDent}.  In Corollary \ref{cordual}, we will show a multi-parameter generalization of the latter two asymptotic results.

\begin{figure}\centering
\includegraphics[width=12cm]{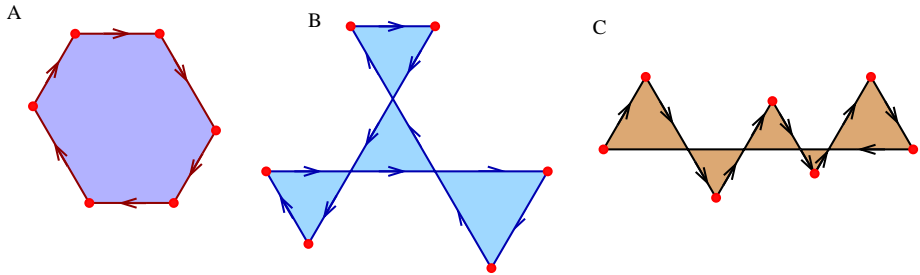}
\caption{Three contours: (A) the contour in MacMahon's theorem, (B) the contour in \cite{CK}, and (C) the contour in \cite{Ciu1}.}\label{contour}
\end{figure}

 We now assume that the set of removed unit triangles is partitioned into $k$ separated clusters (i.e. chains of contiguous unit triangles). Denote these clusters by $C_1,C_2,\dotsc,C_k$ and the distances between them by $d_1,d_2,\dotsc,d_{k-1}$ ($d_i >0$), as they appear from left to rights.  For the sake of convenience, we assume that $C_1$ is attached to the west vertex of the hexagon, that $C_k$ is attached to the east vertex of the hexagon, and that $C_1$ and $C_k$ may be empty.  We use the notation $H_{x,y}(C_1,\dots,C_k; d_1,\dots,d_{k-1})$ for these regions (see Fig. \ref{fern} for an example; the black unit triangles indicate the ones removed). For each cluster $C_i$, we use the notations $U_i$ and $D_i$ for the index sets of its up-pointing and down-pointing triangles.  For each cluster $C_i$, we can shuffle unit triangles at the positions in $U_i \Delta D_i$ to obtain a new cluster $C'_i$, for $i=1,2,\dots,k$. The index sets of triangles in $C'_i$ are denoted by $U'_i$ and $D'_i$. Assume that $|U_i|=u_i$, $D_i=d_i$, $|U'_i|=u'_i$, $|D_i|=d'_i$ and $|U_i\cup D_i|=|U'_i\cup D'_i|=f_i$. We call $f_i$ the \emph{length} of the cluster $C_i$ (and also the length of $C'_i$).
 


 We now consider the behavior of the tiling number of the region when the side-lengths of the outer hexagon and the distances between two consecutive clusters get large.
\begin{cor}\label{cordual}For nonnegative integers $x$ and $y$,
\begin{align}\label{dualeq}
\lim_{N\to \infty}&\frac{\M(H_{Nx,Ny}(C_1,\dots,C_k; Nd_1,\dots,Nd_{k-1}))}{\M(H_{Nx,Ny}(C'_1,\dots,C'_k; Nd_1,\dots,Nd_{k-1}))}\notag\\
&=\prod_{i=1}^{k}\frac{s^+(C_i)s^-(C_i)}{s^+(C'_i)s^-(C'_i)}
\end{align}
 where $s^+(C_i)=\M(S_{u_i,f_i-u_i}(U_i))$ and $s^-(C_i)=\M(S_{d_i,f_i-d_i}(D_i))$ are respectively the tiling numbers of the dented semihexagons whose dents are defined by the up-pointing triangles and down-pointing triangles in the cluster $C_i$, and where $s^+(C'_i)$ and $s^-(C'_i)$ are defined similarly with respect to $C'_i$.
\end{cor}

Corollary \ref{cordual} can be visualized as in Fig. \ref{geointer}, for $k=3$. The dented semihexagons corresponding to $s^+(C_i)$ and $s^-(C_i)$ are the upper and lower halves of the `numerator hexagon' in the $i$th fraction on the right-hand side; the dented semihexagons corresponding to $s^+(C'_i)$ and $s^-(C'_i)$ are the upper and lower halves of the `denominator hexagon' in the $i$th fraction, for $i=1,2,3$.

\begin{figure}\centering
\includegraphics[width=14cm]{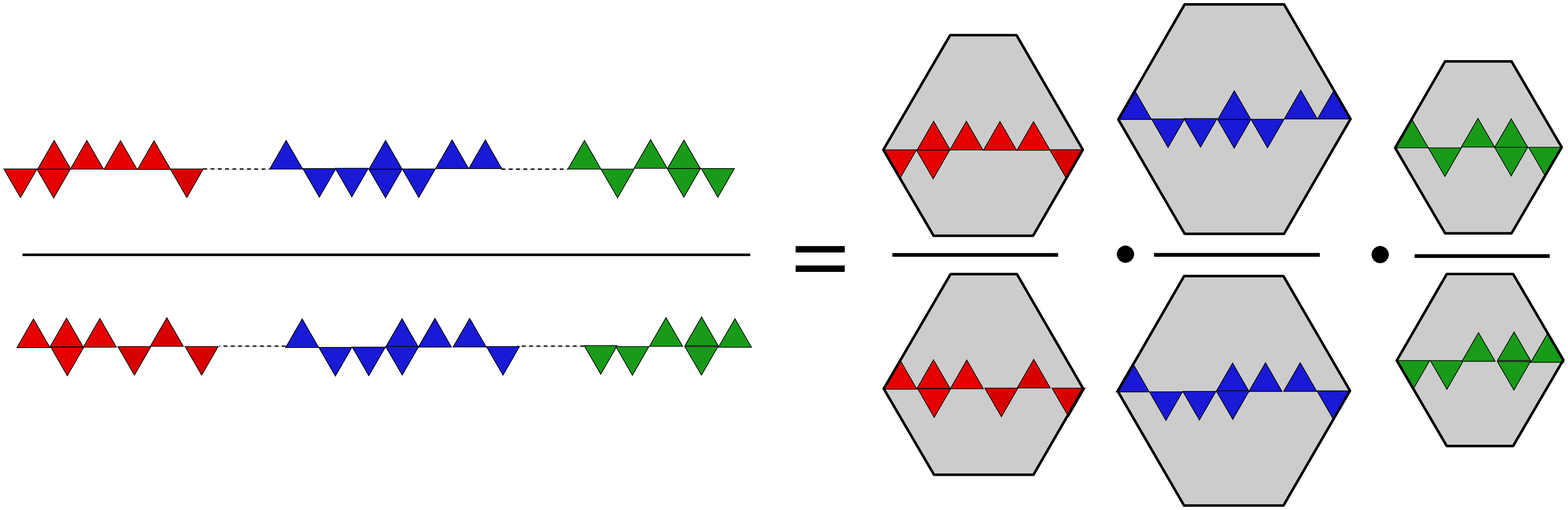}
\caption{Illustrating Corollary \ref{cordual}.}\label{geointer}
\end{figure}

\begin{proof}[Sketch of the proof]
Assume that $U_i=\{a^{(i)}_1,\dots,a^{(i)}_{u_i}\}$ and 
 $U'_i=\{e^{(i)}_1,\dots,e^{(i)}_{u_i}\}$. 
  Applying Theorem \ref{factorization} to the regions $H_{Nx,Ny}(U; D)$ and $H_{Nx,Ny}(U'; D')$, for $U:=\bigcup_iU_i$, $D:=\bigcup_jD_j$, $U':=\bigcup_iU'_i$, and $D':=\bigcup_jD'_j$, we get
\begin{equation}
\frac{\M(H_{Nx,Ny}(C_1,\dots,C_k; Nd_1,\dots,Nd_{k-1}))}{\M(H_{Nx,Ny}(C'_1,\dots,C'_k; Nd_1,\dots,Nd_{k-1}))}=\frac{\Delta(U)\Delta(D)}{\Delta(U')\Delta(D')},
\end{equation}
where the operation $\Delta$ is defined as $\Delta(S):=\prod_{1\leq i<j \leq n}(s_j-s_i)$ for an ordered set $S=\{s_1<s_2<\cdots<s_n\}$. However, we find it more convenient to observe that $\Delta(S)^2=\prod_{1\leq i\not=j \leq n}|s_j-s_i|$ here. 

We have $\frac{\Delta(U)^2}{\Delta(U')^2}=\prod_{i,j}\prod_{p,q}\frac{|a^{(i)}_p-a^{(j)}_q|}{|e^{(i)}_p-e^{(j)}_q|}$, where $p\not=q$ if $i=j$. It is easy to see that if $i\not=j$ the fraction $\frac{|a^{(i)}_p-a^{(j)}_q|}{|e^{(i)}_p-e^{(j)}_q|}$ tends to $1$, as $N$ gets large (for any $p,q$). Thus $\frac{\Delta(U)^2}{\Delta(U')^2}$ tends to $\prod_{i=1}^{k}\prod_{p\not=q}\frac{|a^{(i)}_p-a^{(i)}_q|}{|e^{(i)}_p-e^{(i)}_q|}=\prod_{i=1}^{k}\frac{\Delta(U_i)^2}{\Delta(U'_i)^2}$.

  Moreover, $\frac{\Delta(U_i)}{\Delta(U'_i)} = \frac{s^{+}(C_i)}{s^{+}(C'_i)}$. Thus, $\frac{\Delta(U)}{\Delta(U')}$ tends to $\prod_{i=1}^{k}\frac{s^+(C_i)}{s^+(C'_i)}$. Similarly, we see that $\frac{\Delta(D)}{\Delta(D')}$ tends to $\prod_{i=1}^{k}\frac{s^-(C_i)}{s^-(C'_i)}$, completing the proof.
\end{proof}

By the same arguments, one would imply a $q$-analog of the above asymptotic result from   $q$-Shuffling Theorem \ref{qfactorization}.

\medskip

We now consider the special case in which $U\cap D=\emptyset$ (i.e. in each cluster $C_i$ we have $U_i \cap D_i=\emptyset$). Each cluster $C_i$ can be partitioned into maximal intervals of triangles of the same orientation (we call each of these intervals an `\emph{up-interval}' or a `\emph{down-interval}' if it consists of up-pointing triangles or down-pointing triangles, respectively). For each cluster $C_i$, we can remove forced vertical lozenges above each up-interval and below each down-interval composed of two or more unit triangles. We obtain a new region with the same tiling number in which each cluster is replaced by a chain of removed equilateral triangles of alternating orientations (see Fig. \ref{fern}; the forced lozenges are colored white). Each such chain of triangles is called a `\emph{fern}' (see e.g. \cite{Ciu1,HoleDent}); the side-lengths of triangles in a fern are equal to the lengths of the intervals of unit triangles of the same orientation in the corresponding cluster. Denote by $E_{x,y}(F_1,\dotsc,F_k;d_1,\dots,d_{k-1})$ the corresponding hexagon with ferns removed (the fern $F_i$ corresponds to the cluster $C_i$; and the fern $F'_i$ corresponds to the cluster $C'_i$). By setting $k=3$, $d_1=d_2$ or $d_1=d_2-1$, and specifying that the cluster $C'_i$ has all its up-pointing triangles on the left and all its down-pointing triangles on the right, for $i=1,2,3$ (equivalently, the fern $F'_i$ consists of two triangles, an up-pointing one followed by a down-pointing one), our Corollary \ref{cordual} implies the first author's work in \cite[Theorem 2.11]{HoleDent}. Similarly, we can recover  the work of Ciucu in \cite[Theorem 1.1]{Ciu1} by further specifying that  $C_1=C_3=\emptyset$ (i.e. we actually have only a non-empty fern $F_2$ in the center of the region). 

\begin{figure}\centering
\includegraphics[width=7cm]{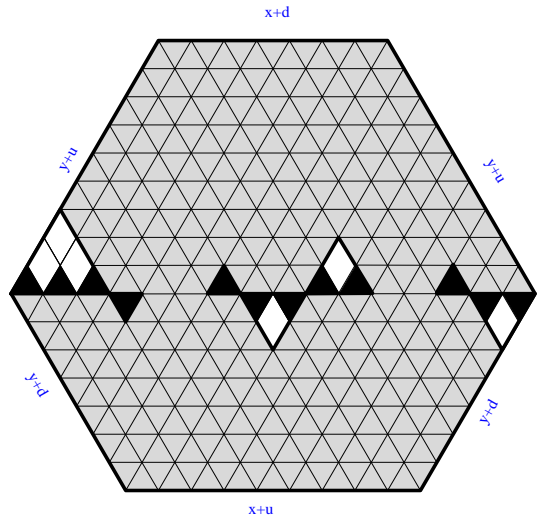}
\caption{Obtained a hexagon with ferns removed from the region $H_{x,y}(C_1,\dots,C_k; d_1,\dots,d_{k-1})$.}\label{fern}
\end{figure}

\section{Proof of the main theorem}
In general, the lozenges in a region can carry weights. In the weighted case, $\M(R)$ denotes the sum of weights of the tilings in $R$, where the \emph{weight} of  a tiling is the product of weights of its constituent lozenges. The removal of one or more forced lozenges changes the number of tilings of the region by a factor equal to the reciprocal of the weighted product of the forced lozenges. In particular, if we remove the lozenges $l_1,l_2,\dots,l_k$ from a region $R$ and get a new region $R'$, then
\begin{equation}
\M(R')=\left(\prod_{i=1}^{k}wt(l_i)\right)^{-1}\M(R),
\end{equation}
where $wt(l_i)$ is the weight of the forced lozenge $l_i$. 

A \emph{(perfect) matching} of a graph is a collection of disjoint edges that covers all the vertices of the graph. When edges of a graph carry weights, we denote by $\M(G)$ the weighted sum of matchings in $G$, where the weight of a matching is the product of weights of its edges. (In the unweighted case, $\M(G)$ is exactly the number of matchings of $G$.)  The \emph{(planar) dual graph} of the region $R$ on the triangular lattice is the graph whose vertices are the unit triangles in $R$ and whose edges connect precisely two unit triangles sharing an edge. The edges of the dual graph inherit the weights from the lozenges of the corresponding region. There is a natural (weight-preserving) bijection between tilings of a region and matchings of its dual graph.

Our proof is based on the following powerful graphical condensation lemma first introduced by Kuo \cite{Kuo}:
\begin{lem}\label{kuothm}
Let $G=(V_1,V_2,E)$ be a (weighted) planar bipartite graph with the two vertex classes $V_1$ and $V_2$ such that $|V_1|=|V_2|+1$. Assume that $u,v,w,s$ are four vertices appearing in this cyclic order around a face of $G$, such that $u,v,w\in V_1$ and $s\in V_2$. Then
\begin{align}
\M(G-v)&\M(G-\{u,w,s\})=\M(G-u)\M(G-\{v,w,s\})+\M(G-w)\M(G-\{u,v,s\}).
\end{align}
\end{lem}

\medskip

If a region admits a lozenge tiling, then it must have the same number of up-pointing and down-pointing unit triangles. We call such a region \emph{balanced}. The following lemma allows us to decompose a region into smaller regions when enumerating tilings in certain situations.

\begin{lem}[Region-splitting Lemma \cite{Tri1,Tri2}]\label{RS}
Let $R$ be a balanced region on the triangular lattice. Assume that a balanced sub-region $Q$ of $R$ satisfies the condition that the unit triangles in $Q$ that are adjacent to some unit triangle of $R-Q$ have the same orientation.
Then $\M(R)=\M(Q)\, \M(R-Q).$
\end{lem}

Given $k$ positive integers $\lambda_1\geq \lambda_2\geq \dots \geq \lambda_k$, a \emph{plane partition} of shape $(\lambda_1,\lambda_2,\dots,\lambda_k)$ is an array of non-negative integers
\begin{center}
\begin{tabular}{rccccccccc}
$n_{1,1}$   &$n_{1,2}$                 &$n_{1,3}$               & $\dotsc$               &  $\dotsc$                        & $\dotsc$                            &   $n_{1,\lambda_1}$ \\\noalign{\smallskip\smallskip}
$n_{2,1}$   &  $n_{2,2}$              & $n_{2,3}$             &  $\dotsc$               & $\dotsc$                        &         $n_{2,\lambda_2}$&          \\\noalign{\smallskip\smallskip}
$\vdots$    &       $\vdots$            & $\vdots$                &        $\vdots$         &     \reflectbox{$\ddots$\quad}               &    &              \\\noalign{\smallskip\smallskip}
 $n_{k,1}$  &  $n_{k,2}$               & $n_{k,3}$              &     $\dotsc$             &   $n_{k,\lambda_k}$ &                                          &           \\\noalign{\smallskip\smallskip}
\end{tabular}
\end{center}
so that $n_{i,j}\geq n_{i,j+1}$ and $n_{i,j}\geq n_{i+1,j}$ (i.e. all rows and all columns are weakly decreasing from left to right and from top to bottom, respectively). The sum of all entries of a plane partition $\pi$ is called the \emph{volume} (or the \emph{norm}) of the plane partition, and denoted by $|\pi|$.

A \emph{column-strict
plane partition} is a plane partition having columns strictly decreasing. A \emph{column-strict plane partition} is a plane partition having columns strictly decreasing. We now assign to each right-tilting lozenge in the dented semihexagon similarly to the upper half of the doubly dented hexagon. In particular, a eight-tilting lozenge is weighted by $q^{z}$, where the distance between its center and the base of the region is $(z-\frac{1}{2})\frac{\sqrt{3}}{2}$.  Denote $\M_q(S_{a,b}(s_1,s_2,\dots,s_a))$ the corresponding tiling generating function of the dented semihexagon $S_{a,b}(s_1,s_2,\dots,s_a)$.There is a well-known (weight preserving) bijection between the lozenge tilings of $S_{a,b}(s_1,s_2,\dots,s_a)$ and the column-strict plane partitions of shape $(s_{a}-a,s_{a-1}-a+1,\dotsc,s_1-1)$ with positive entries at most $a$ (see e.g. \cite{CLP} and \cite{Car}). The following weighted enumeration of the dented semihexagon follows directly from the bijection and equation (7.105) in \cite[pp. 375]{Stanley}.

\begin{lem}\label{qCLP} For nonnegative $a,b$ and positive integers $1\leq s_1<s_2<\cdots<s_a$
\begin{equation}
\M_q(S_{a,b}(s_1,s_2,\dots,s_a))=\sum_{\pi}q^{|\pi|}=q^{\sum_{i=1}^a(s_i-i)}\prod_{1\leq i <j \leq q}\frac{q^{s_j}-q^{s_i}}{q^{j}-q^{i}},
\end{equation}
where the sum after the first equality sign is taken over all column-strict plane partitions $\pi$ of shape $(s_{a}-a,s_{a-1}-a+1,\dotsc,s_1-1)$ with positive entries at most $a$.
\end{lem}

\begin{proof}[Proof of Theorem \ref{qfactorization}]
Denote by $g_{x,y}(U;D;U';D')$ the right-hand side of (\ref{qmianeq})\footnote{As the right-hand side of (\ref{qmianeq}) does not depend on the index $B$, our $g$-function does not depend on $B$.}, i.e.
\begin{equation}
g_{x,y}(U;D;U';D')=q^C \frac{\displaystyle\prod_{1\leq i <j\leq u}\frac{q^{s_j}-q^{s_i}}{q^{j}-q^{i}}\displaystyle\prod_{1\leq i <j\leq d}\frac{q^{t_j}-q^{t_i}}{j-i}\PP_q(u,d,y)}{\displaystyle\prod_{1\leq i <j\leq u'}\frac{q^{s'_j}-q^{s'_i}}{q^{j}-q^{i}}\displaystyle\prod_{1\leq i <j\leq d'}\frac{q^{t'_j}-q^{t'_i}}{q^{j}-q^{i}}\PP_q(u'd',y)},
\end{equation}
we need to show that
\begin{equation}\label{refineeq}
\M_q(H_{x,y}(U;D;B))=g_{x,y}(U;D;U';D')\M_q(H_{x,y}(U';D';B)).
\end{equation}

\medskip

We prove (\ref{refineeq}) by induction on $x+y$. The base cases are the situations when $x=b$ and when $y=0$.

If $y=0$, then we apply Region-splitting Lemma \ref{RS} to the region $R=H_{x,0}(U;D;B)$ with the subregion $Q$ the portion above the horizontal axis $l$. The subregion $Q$ is congruent with the dented semihexagons $S_{u,x+n-u}(U)$ weighted as in Lemma \ref{qCLP}, and  the  $R-Q$, after rotated $180^{\circ}$, is congruent with $S_{d,x+n-d}(r(D))$ weighted similarly with $q$ replaced by $q^{-1}$ (see Fig. \ref{basecase} A). Here we use the notation $r(S)$ for the reflection of the index set $S$, i.e. the set $(x+y+n+1)-S=\{(x+y+n+1)-s: s \in S\}.$
We have
\begin{align}
\M_q(H_{x,0}(U;D;B))=\M_q(S_{u,x+n-u}(U))\M_{q^{-1}}(S_{d,x+n-d}(r(D))).
\end{align}
 Similarly, the tiling number of $R'=H_{x,0}(U';D';B)$ is also written by a product of tiling numbers of two dented semihexagons as
 \begin{align}
\M_q(H_{x,0}(U';D';B))=\M_q(S_{u,x+n-u}(U'))\M_{q^{-1}}(S_{d,x+n-d}(r(D'))),
\end{align}
and (\ref{refineeq}) follows from Lemma \ref{qCLP}.
 
If $x=b$, consider the subregion $Q$ of  $R=H_{b,y}(U;D;B)$ that is obtained from the portion above the axis $l$ by removing all up-pointing unit triangles in $(U\cup D\cup B)^c$ (we note that in this case $|(U\cup D\cup B)^c|=y$). $Q$ is the dented semihexagon $S_{y+u,b+n-u}((U\cup D\cup B)^c\cup U)$ weighted as in Lemma \ref{qCLP}, and its complement, after removing forced lozenges at the positions in $(U\cup D\cup B)^c$ and rotating $180^{\circ}$, is the dented semihexagon  $S_{y+d,b+n-u}((U\cup D\cup B)^c\cup D)$ weighted similarly with $q$ replaced by $q^{-1}$ (see Fig. \ref{basecase} B). This way, the tiling number of $R$ is written as the product of tiling numbers of two dented semihexagons. Similarly, the tiling number of $R'=H_{b,y}(U';D';B)$ is also written by a product of tiling numbers of two dented semihexagons. Then (\ref{refineeq}) also follows from  Lemma \ref{qCLP}.

For the induction step, we assume that $x>b$, $y>0$, and that (\ref{refineeq}) holds for any $H$-type regions whose sum of $x$- and $y$-parameters is strictly less than $x+y$. We will use Kuo condensation in Theorem \ref{kuothm} to obtain a recurrence for the tiling generating function on the left-hand side of (\ref{refineeq}), and we show that the expression on the right-hand side satisfies the same equation. Then (\ref{refineeq})  follows from the induction principle.

We apply Kuo condensation to the \emph{dual graph} $G$ of the region $R$ obtained from $H_{x,y}(U;D;B)$ by adding a layer of unit triangle on the top of the hexagon, with the four vertices $u,v,w,s$ as in Fig.  \ref{hexkuo1} (the region restricted by the bold contour indicates the region $H_{x,y}(U;D;B)$). In particular, the vertices $w$ and $v$ correspond to the up-pointing triangles at the first and the last positions in $(U\cup D\cup B)^c$, the vertex $v$ corresponds to the up-pointing triangle on the northeast corner of the region, and the vertex $s$ corresponds to the down-pointing triangle on the southeast corner of the region.

\begin{figure}\centering
\includegraphics[width=12cm]{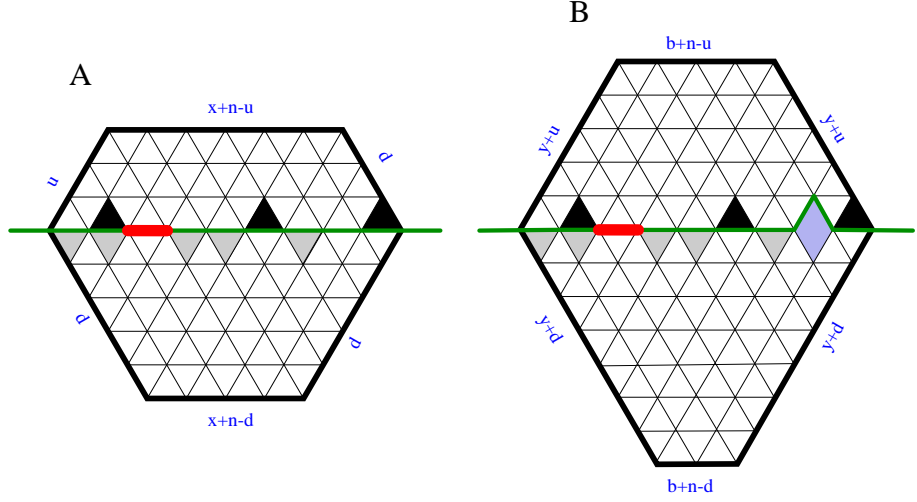}
\caption{Applications of the Region-splitting Lemma \ref{RS} in the base cases: (A) $y=0$, (B) $x=b$.}\label{basecase}
\end{figure}
\begin{figure}\centering
\includegraphics[width=8cm]{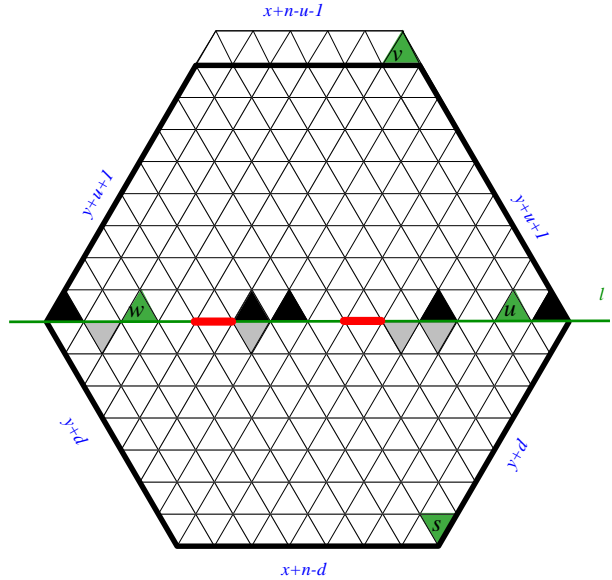}
\caption{Applying Kuo condensation to a hexagon.}\label{hexkuo1}
\end{figure}
\begin{figure}\centering
\includegraphics[width=12cm]{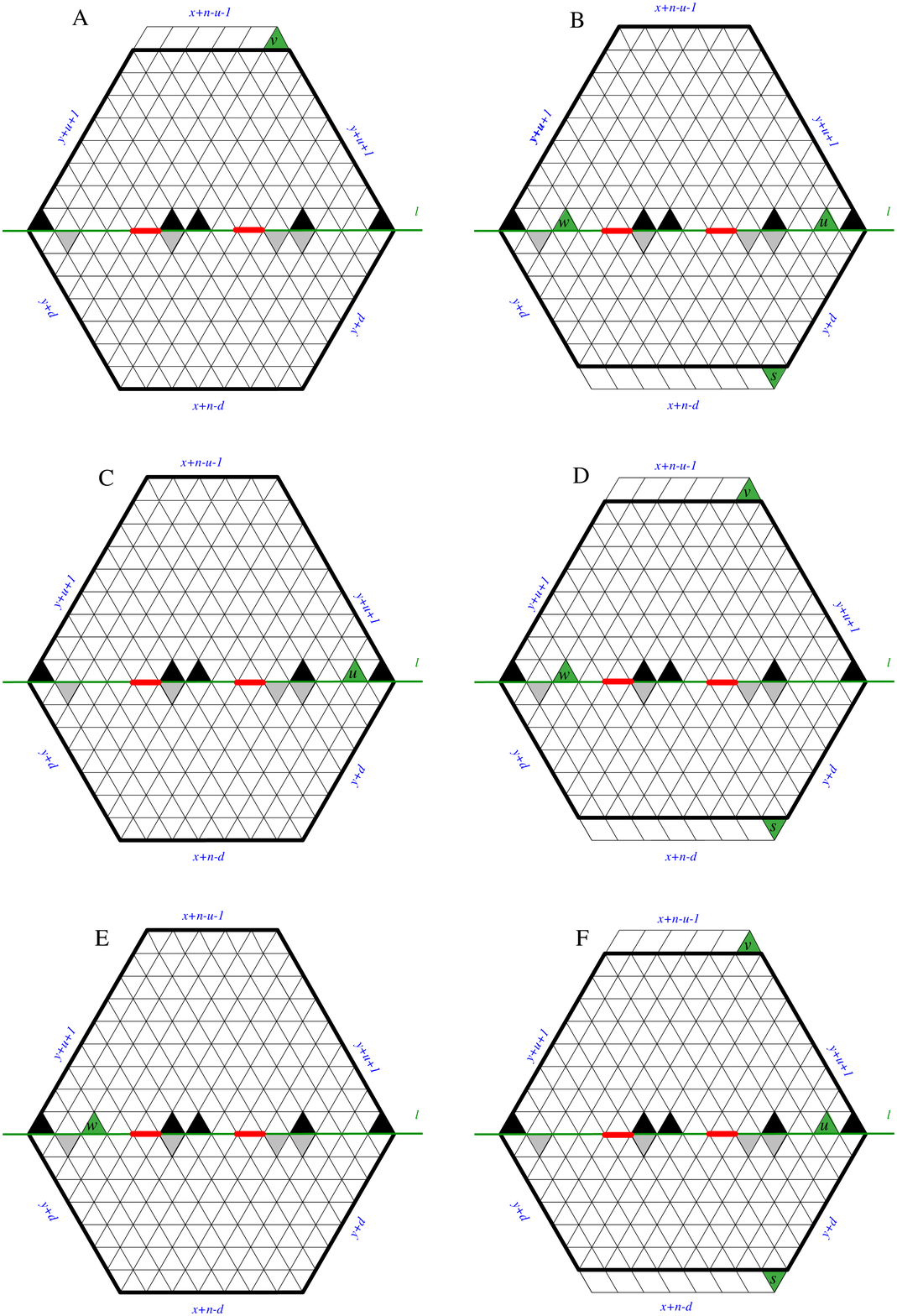}
\caption{Obtaining the recurrence for the tiling numbers.}\label{hexkuo2}
\end{figure}

We consider the region corresponding to $G-v$ (i.e., the region obtained from $R$ by removing the $v$-triangle as in Fig.  \ref{hexkuo2} A). The removal of the $v$-triangle yields several forced lozenges along the top of the region. After removing these forced lozenges (this lozenge removal changes the tiling number of the region by a factor $q^{(y+u+1)(x+n-u-1)}$), we get back the region $H_{x,y}(U;D;B)$. This means that we get
\begin{equation}\label{kuoeqa}
\M(G-v)=q^{(y+u+1)(x+n-u-1)}\M_q(H_{x,y}(U; D;B)).
\end{equation}
 By considering forced lozenges in the regions corresponding  to the graphs $G-\{u,w,s\},$  $G-u,$  $G-\{v,w,s\},$  $G-w,$ and  $G-\{u,v,s\}$ (as shown in Fig. \ref{hexkuo2} B -- F, respectively), we get
\begin{equation}\label{kuoeqb}
\M(G-\{u,w,s\})=\M_q(H_{x-1,y-1}(\alpha \beta U; D;B)),
\end{equation}
\begin{equation}\label{kuoeqc}
\M(G-u)=\M_q(H_{x-1,y}(\beta U;D;B)),
\end{equation}
\begin{equation}\label{kuoeqd}
\M(G-\{v,w,s\})=q^{(y+u+1)(x+n-u-1)}\M_q(H_{x,y-1}(\alpha U;D;B)),
\end{equation}
\begin{equation}\label{kuoeqe}
\M(G-w)=\M_q(H_{x-1,y}(\alpha U;D;B)),
\end{equation}
\begin{equation}\label{kuoeqf}
\M(G-\{u,v,s\})=q^{(y+u+1)(x+n-u-1)}\M_q(H_{x,y-1}(\beta U;D;B)),
\end{equation}
where we use the notations $\alpha U$, $\beta U$ and $\alpha \beta U$ for the unions $U\cup\{\alpha \}$, $U\cup\{\beta \}$ and $U\cup\{\alpha,\beta\}$, respectively.
Plugging Eqs. \ref{kuoeqa}--\ref{kuoeqf} into the equation in Kuo's Lemma \ref{kuothm}, we get the recurrence (all powers of $q$ cancel out):
\begin{align}\label{recurrence1}
\M_q(H_{x,y}(U;D;B))\M_q(H_{x-1,y-1}(\alpha \beta U; D;B))=&\M_q(H_{x-1,y}(\beta U;D;B))\M_q(H_{x,y-1}(\alpha U;D;B))\notag\\
&+\M_q(H_{x-1,y}(\alpha U;D;B))\M_q(H_{x,y-1}(\beta U;D;B)).
\end{align}



\bigskip

Denote by $h_{x,y}(U;D;U';D'):=g_{x,y}(U;D;U';D')\M(H_{x,y}(U';D';B))$ the expression on the right-hand side of (\ref{refineeq}). 
To finish the proof, we shall show that the $h_{x,y}(U;D;U';D')$ also satisfies recurrence (\ref{recurrence1}). Equivalently, we need to verify
\begin{align}\label{hrecur}
\frac{h_{x-1,y}(\beta U;D; \beta U';D')h_{x,y-1}(\alpha U;D; \alpha U';D'))}{h_{x,y}(U;D;U';D')h_{x-1,y-1}(\alpha\beta U; D; \alpha\beta U'; D')}+\frac{h_{x-1,y}(\alpha U;D;\alpha U';D')h_{x,y-1}(\beta U;D;\beta U';D'))}{h_{x,y}(U;D;U';D')h_{x-1,y-1}(\alpha\beta U; D;\alpha\beta U'; D')}=1.
\end{align}
We claim
\begin{clm}
\begin{equation}\label{ratioeq1}
\frac{g_{x-1,y}(\beta U;D; \beta U';D')g_{x,y-1}(\alpha U;D; \alpha U';D'))}{g_{x,y}(U;D;U';D')g_{x-1,y-1}(\alpha\beta U; D; \alpha\beta U'; D')}=1
\end{equation}
and
\begin{equation}\label{ratioeq2}
\frac{g_{x-1,y}(\alpha U;D;\alpha U';D')g_{x,y-1}(\beta U;D;\beta U';D'))}{g_{x,y}(U;D;U';D')g_{x-1,y-1}(\alpha\beta U; D;\alpha\beta U'; D')}=1.
\end{equation}
\end{clm}
\begin{proof}
We have from the definition
\begin{align}\label{PPeq}
\frac{\PP_q(u,d,y)}{\PP_q(u',d',y)}=q^{-udy+u'd'y}\frac{\Delta_q([u])\Delta_q([d])}{\Delta_q([u'])\Delta_q([d'])}\frac{\Delta_q([y+u'])\Delta_q([y+d'])}{\Delta_q([y+u])\Delta_q([y+d])},
\end{align}
where, for any ordered index set $S=\{s_1,s_2,\dots,s_k\}$, we define $\Delta_q(S):=\prod_{1\leq i < j\leq k} q^{s_j}-q^{s_i}$.
Therefore, we can rewrite the $g$-function as:
\begin{align}
g_{x,y}(U;D;U';D')=q^{(d-x-n)\binom{y+d+1}{2}-(d'-x-n)\binom{y+d'+1}{2}}\frac{\Delta_q(U)\Delta_q(D)}{\Delta_q(U')\Delta_q(D')}\frac{\Delta_q([y+u'])\Delta_q([y+d'])}{\Delta_q([y+u])\Delta_q([y+d])}.
\end{align}
It is easy to see that terms correspond to the faction  $\frac{\Delta_q([y+u'])\Delta_q([y+d'])}{\Delta_q([y+u])\Delta_q([y+d])}$ cancel out in  (\ref{ratioeq1}). The exponents of $q$ also cancel out easily. 

Moreover, since the position sets of down-pointing triangles do not change, the corresponding $\Delta_q(D)$  and $\Delta_q(D')$ terms canceled out. This makes  Eq. (\ref{ratioeq1}) become
\begin{equation}
\frac{\Delta_q(\beta U)}{\Delta_q(\beta U')}\frac{\Delta_q(\alpha U)}{\Delta_q(\alpha U')} = \frac{\Delta_q(U)}{\Delta_q(U')}\frac{\Delta_q(\alpha\beta U)}{\Delta_q(\alpha\beta U')}.
\end{equation}
Dividing two sides of the above equation by $\frac{\Delta_q(U)^2}{\Delta_q(U')^2}$, we get both equal to:
\begin{align}
\frac{\prod_{i=1}^{u}|\beta-s_i|_q}{\prod_{i=1}^{u'}|\beta-s'_i|_q}\frac{\prod_{i=1}^{u}|\alpha-s_i|_q}{\prod_{i=1}^{u'}|\alpha-s'_i|_q},
\end{align}
where the \emph{`$q$-absolute value'} $|x-y|_q$ is define to be $q^{x}-q^{y}$ if $x\geq y$, and is $q^{y}-q^{x}$ if $y>x$.
This implies (\ref{ratioeq1}). The identity (\ref{ratioeq2}) follows from (\ref{ratioeq1}) by interchanging the roles of $\alpha$ and $\beta$.
\end{proof}

By (\ref{ratioeq1}) and (\ref{ratioeq2}) in the above claim, we have (\ref{hrecur}) simplified as
\begin{align}
\frac{\M_q(H_{x-1,y}(\beta U';D';B))\M_q(H_{x,y-1}(\alpha U';D';B))}{\M_q(H_{x,y}(U';D';B))\M_q(H_{x-1,y-1}(\alpha \beta U'; D';B))}+\frac{\M_q(H_{x-1,y}(\alpha U';D;B'))\M_q(H_{x,y-1}(\beta U';D';B))}{\M_q(H_{x,y}(U';D';B))\M_q(H_{x-1,y-1}(\alpha \beta U'; D';B))}=1,
\end{align}
which is obtained from the application of recurrence (\ref{recurrence1}) to the region $H_{x,y}(U';D';B)$. We just verified that the function $h_{x,y}(U;D;U';D')$ (i.e. the expression on the right-hand side of (\ref{refineeq})) satisfies  (\ref{recurrence1}). This finishes the proof.
\end{proof}

\section{Generalizations, more applications, and future directions}

(1). Recall that when $U\cap D=\emptyset$ and $B=\emptyset$, our region (written in `cluster form') $H_{x,y}(U; D)=H_{x,y}(U; D; \emptyset)=H_{x,y}(C_1,\dotsc,C_k; d_1,\dots,d_{k-1})$ is tiling-equinumerous with a hexagon with ferns removed $E_{x,y}(F_1,\dotsc,F_k;d_1,\dots,d_{k-1})$.  Viewed in this way, our Theorem \ref{genfactorization} implies a number of known enumerations about regions with ferns removed. The general idea is that we choose suitable position sets $U,D,U',D'$ so that the region $H_{x,y}(U; D)$ in Eq. (\ref{genmaineq}) is the one that we want enumerate, and the region $H_{x,y}(U';D')$ is a known region up to removal of forced lozenges. Let us first consider the implication to  \cite[Theorem 2.12]{HoleDent} as follows. Suppose we specialize our region as follows:
\begin{enumerate}[nolistsep]
\item $k=3$,
\item $C_1$ and $C_3$ contain the same number of unit triangles,
\item $|d_1-d_2|\leq 1$ (i.e. $C_2$ is centered between $C_1$ and $C_3$, or as close to centered as possible)
\item all unit triangles in $C'_1$ have the same orientation,
\item all unit triangles in $C'_3$ have opposite orientation to those in $C'_1$, and
\item all unit triangles in $C'_2$ are up-pointing.
\end{enumerate}
In this case, the region $H_{x,y}(U; D)$ becomes a hexagon with three ferns removed and $H_{x,y}(U';D')$, after the removal of forced lozenges, becomes a \emph{cored hexagon} as in \cite{CEKZ} (see the illustration in Figure \ref{forced}). Therefore, our Theorem \ref{genfactorization} implies the enumeration of tilings of hexagons with three ferns removed in \cite[Theorem 2.12]{HoleDent}.

\medskip

By specifying that $C_1=C'_1$ and $C_3=C'_3$, in addition to the six conditions above, we recover Ciucu's Theorem 2.1 in \cite{Ciu1} about $F$-cored hexagons (after removing forced lozenges, the region $H_{x,y}(U;D)$ becomes an $F$-cored hexagon and $H_{x,y}(U';D')$, as in the previous case, becomes a cored hexagon). 

\medskip

Similarly, if we require
\begin{enumerate}[nolistsep]
\item $k=2$,
\item all unit triangles in $C'_1$ have the same orientation, and
\item all unit triangles in $C'_2$ have opposite orientation to those in $C'_1$, 
\end{enumerate}
then we get Theorem 1.1 in \cite{Twofern} about hexagons with two ferns removed, called `\emph{doubly--intruded hexagons}' (the region $H_{x,y}(U;D)$ is now a doubly--intruded hexagon and $H_{x,y}(U';D')$ becomes a hexagon in MacMahon's theorem, after removing forced lozenges).

\begin{figure}\centering
\includegraphics[width=12cm]{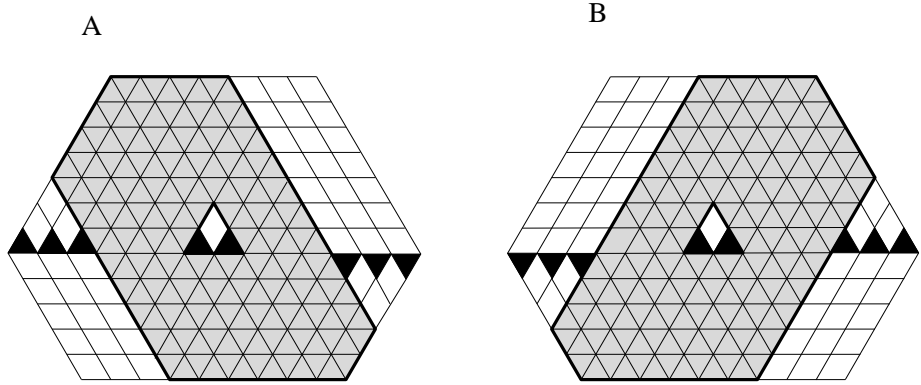}
\caption{Obtaining cored hexagons from hexagons with three ferns removed.}\label{forced}
\end{figure}

\medskip

(2). By equation (7.105) of \cite{StanleyB}, we have\footnote{The notation $\textbf{1}^n$ in the argument of a Schur function stands for $n$ arguments equal to 1.}
\begin{equation}
\prod_{1\leq i< j \leq u}\frac{s_j-s_i}{j-i}=s_{\lambda(\{s_1,\dots,s_u\})}(\textbf{1}^{u}),
\end{equation}
where the partition $\lambda(\{s_1,\dots,s_u\}):=(s_u-u+1,\dotsc,s_2-1,s_1)$. On the other hand, it is not hard to see that we also have
\begin{align}
\M\left(H_{x,y}(U;D)\right)
=\sum_{|S|=y}s_{\lambda(U\cup S)}(\textbf{1}^{u+y})s_{\lambda(D \cup S)}(\textbf{1}^{d+y}),
\end{align}
where the sum runs over all $y$-subsets $S$ of $ [x+y+n]-(U\cup D)$. Indeed, this follows from the fact that in each tiling, precisely $y$ of the $x+y$ unit segments on the lattice line from along which we removed the $n$ unit triangles are straddled by vertical lozenges. Hence one can write Eq. \ref{maineq1} in Theorem \ref{factorization} as
\begin{equation}\label{schureq}
\frac{  \sum_{|S|=y}\textbf{s}_{\lambda(U\cup S)}(\textbf{1}^{u+y})\textbf{s}_{\lambda(D \cup S)}(\textbf{1}^{d+y})}{  \sum_{|S|=y}\textbf{s}_{\lambda(U'\cup S)}(\textbf{1}^{u+y})\textbf{s}_{\lambda(D' \cup S)}(\textbf{1}^{d+y})}= \frac{\textbf{s}_{\lambda(U)}(\textbf{1}^{u})\textbf{s}_{\lambda(D)}(\textbf{1}^{d})}{\textbf{s}_{\lambda(U')}(\textbf{1}^{u})\textbf{s}_{\lambda(D')}(\textbf{1}^{d})},
  \end{equation}
where the sum is taken over all $y$-subsets $S$ of $ [x+y+n]-(U\cup D)$.

We note that $q$-Shuffling Theorem \ref{qfactorization} gives further supporting evidence for the existence of a Schur function identity behind Shuffling Theorem \ref{factorization}, as we have
 \[\prod_{1\leq i <j\leq u}\frac{1-q^{s_j-s_i}}{1-q^{j-i}}=q^{A}\, \textbf{s}_{\lambda(\{s_1,\dotsc,s_u\})}(q,q^2,q^3,\dots),\]
  for some constant $A$. This would imply that Eq. \ref{schureq} is still true (up to a $q$-power) when $\textbf{1}^{n}$ is replaced by the sequence $(q,q^2,q^3,\dots,q^{n})$.

It would be interesting to know if the following general sum has a similar simplification:
\begin{equation}\label{schureq2}
\frac{  \sum_{|S|=y}\textbf{s}_{\lambda(U\cup S)}(\textbf{X}^{m+y})\textbf{s}_{\lambda(D \cup S)}(\textbf{X}^{m+(d-u)+y})}{  \sum_{|S|=y}\textbf{s}_{\lambda(U'\cup S)}(\textbf{X}^{m+y})\textbf{s}_{\lambda(D' \cup S)}(\textbf{X}^{m+(d-u)+y})},
   \end{equation}
where $\textbf{X}^{m}$ denotes the sequence of variables $x_1,x_2,\dotsc,x_m$.

(3). Motivated by Stanley's classical paper \cite{Stanley} on symmetric plane partitions, we would like to investigate symmetric tilings of  $H_{x,y}(U;D)$. There are two natural classes of symmetric tilings: the tilings which are invariant under a reflection over a vertical axis, and those which are invariant under a $180^{\circ}$ rotation (these tilings correspond to the transposed-complementary and self-complementary plane partitions). The shuffling theorems for these symmetry classes will be investigated in separate papers \cite{shuffling2, shuffling3}.

\section*{Acknowledgement}
The authors would like to thank Brendon Rhoades for fruitful discussions. The first author (T.L.) was supported in part by Simons Foundation Collaboration Grant \#-585923.

\end{document}